\documentclass[12pt]{amsart}
\usepackage{fullpage}   
\usepackage{amssymb}    
\usepackage{amsmath}
\usepackage{amsthm}
\usepackage{array}
\usepackage{enumitem}
\usepackage{url}
\usepackage{verbatim}
\usepackage[all]{xy}
\usepackage{multicol}
\usepackage{subfigure}
\usepackage{lscape}
\usepackage{mathtools}

\theoremstyle{plain}
\newtheorem{theorem}{Theorem}[section]
\newtheorem{lemma}[theorem]{Lemma}
\newtheorem{definition}[theorem]{Definition}
\newtheorem{proposition}[theorem]{Proposition}
\newtheorem{corollary}[theorem]{Corollary}

\theoremstyle{remark}
\newtheorem{example}[theorem]{Example}
\newtheorem{remark}[theorem]{Remark}
\newtheorem*{acknowledgment}{Acknowledgement}

\numberwithin{equation}{section}

\newcommand{\seclabel}[1]{\label{sec:#1}}   

\newcommand{\Inn}{\mathrm{Inn}}

\title{Semiautomorphic Inverse Property Loops}

\author[M. Greer]{Mark Greer}

\address{Department of Mathematics\\
One Harrison Plaza \\
University of North Alabama\\
Box 5051\\
Florence, AL 35632 USA }

\email{\url{mgreer@una.edu}}

\begin{document}
\allowdisplaybreaks[4]

\begin{abstract}
We define a variety of loops called semiautomorphic, inverse property loops that generalize Moufang and Steiner loops.  We first show an equivalence between a previously studied variety of loops.  Next we extend several known results for Moufang and Steiner loops.  That is, the commutant is a subloop and if $a$ is in the commutant, then $a^{2}$ is a Moufang element, $a^{3}$ is a $c$-element and $a^{6}$ is in the center.  Finally, we give two constructions for semiautomorphic inverse property loops based on Chein's and  de Barros and Juriaans' doubling constructions.
\end{abstract}


\maketitle

\section{Introduction}
\seclabel{intro}
A loop $(Q,\cdot)$ consists of a set $Q$ with a binary operation $\cdot : Q\times Q\to Q$ such that $(i)$ for all $a,b\in Q$, the equations $ax = b$ and $ya = b$ have unique solutions $x,y\in Q$, and $(ii)$ there exists $1\in Q$ such that $1x = x1 = x$ for all $x\in Q$.  Standard references for loop theory are \cite{bruck, hala}.  In a loop \emph{Q}, the left and right translations by $x \in Q$ are defined by $yL_{x} = xy$ and $yR_{x} = yx$ respectively.  We define the \emph{multiplication group} of $Q$, Mlt($Q$)$=\left\langle R_{x},L_{x}\mid x\in Q\right\rangle$.  Similarly, we define the \emph{inner mapping group} of $Q$, Inn($Q$)$=$Mlt$_{1}$($Q$)$=$ $\{\theta\in $Mlt($Q$)$\mid 1\theta=1\}$.  

In general, the inner mappings of a nonassociative loop are not automorphisms of the loop (except in the class of \emph{automorphic} loops which are defined
by that very property). However, in some of the various classes of loops which are commonly studied, the action of the inner mapping group still preserves some of the loop structure.

\emph{Moufang loops}, which are easily the most studied class of loops, are defined by the identity $(xy)(zx) = x((yz)x)$ (or other identities equivalent to this). Every inner mapping $\theta$ of a Moufang loop $Q$ is a \emph{semiautomorphism}, that is, $1\theta = 1$ and
\[
(xyx)\theta = x\theta \cdot y\theta \cdot x\theta
\]
for all $x,y\in Q$. (Since Moufang loops are \emph{flexible}, that is, $(xy)x=x(yx)$ for all $x,y$, we may write $xyx$ unambiguously.)

\emph{Steiner loops},  which arise from Steiner triple systems, are loops satisfying the identities $xy=yx$, $x(yx)=y$. Every inner mapping $\theta$ of a Steiner loop is also a semiautomorphism:  $(xyx)\theta = y\theta = x\theta\cdot y\theta\cdot x\theta$.

In this paper, we focus on this property of inner mappings to study a class of loops generalizing both Moufang loops and Steiner loops.

\begin{definition}
A loop $Q$ is said to be a \emph{semiautomorphic, inverse property loop} (or just semiautomorphic IP loop) if
\begin{enumerate}
\item $Q$ is flexible, that is, $(xy)x=x(yx)$ for all $x,y\in Q$;
\item $Q$ has the \emph{inverse property (IP)}, that is, for each
$x\in Q$, there exists $x^{-1}\in Q$ such that $x^{-1}(xy) = y$ and $(yx)x^{-1} = y$ for all $y\in Q$:
\item Every inner mapping is a semiautomorphism, that is, for each $\theta\in Inn(Q)$,
$x\theta\cdot y\theta\cdot x\theta = (x\cdot y\cdot x)\theta$ for all $x,y\in Q$.
\end{enumerate}
\end{definition}

\begin{remark}
We could have dispensed with flexibility as part of the definition and simply fixed a convention for what a semiautomorphism is, such as $x\theta\cdot (y\theta\cdot x\theta) = (x\cdot (y\cdot x))\theta$. However, it is easy to show that flexibility is then a consequence.
\end{remark}

If $\theta$ is a semiautomorphism of a flexible loop $Q$, then for all $x\in Q$, $x\theta = (x x^{-1} x)\theta = x\theta\cdot x^{-1}\theta\cdot x\theta$, and canceling gives $1 = x\theta \cdot x^{-1}\theta$. Thus if we define the inversion map $J : Q\to Q$ by $xJ = x^{-1}$, we have $\theta^J = \theta$ for any semiautomorphism $\theta$.

It follows that any semiautomorphic IP loop is an example of a variety of loops which have already appeared in the literature called ''$J$-loops" or ``RIF loops'' (RIF = \textbf{R}espects \textbf{I}nverses and \textbf{F}lexible). $J$-loops were introduced in \cite{ikuta} and RIF loops were introduced in \cite{KKP}.  Commutative RIF loops were studied in \cite{KV09}.  Recalling that a loop is \emph{diassociative} if any subloop generated by at most two elements is associative, we have the following, which follows from the main result of \cite{KKP}.

\begin{proposition}\cite{KKP}.
Every semiautomorphic IP loop is diassociative.
\end{proposition}
\begin{remark}
Throughout, we will make explicit use of diassociativity for simplifications, without reference.
\end{remark}

Our first main result, proved in {\S}\ref{sect2} is the converse of our observation that every semiautomorphic IP loop is a RIF loop. We state this as the following characterization (eschewing the somewhat cryptic ``RIF'' terminology).

\begin{theorem}
Let $Q$ be a loop. The following are equivalent.
\begin{enumerate}
\item $Q$ is a semiautomorphic IP loop;
\item $Q$ is a flexible IP loop such that $\theta^J = \theta$ for all $\theta\in Inn(Q)$.
\end{enumerate}
\label{samedef}
\end{theorem}
 
The \emph{commutant} of a loop $Q$ is the set $C(Q) = \{ a\in Q\ |\ ax=xa\ \forall x\in Q\}$.  In general, the commutant of a loop is not a subloop, although it is known to be so in certain cases, such as for Moufang loops. In {\S}\ref{sect3}, we study the commutant of a semiautomorphic IP loop and show that it is a subloop (Theorem \ref{SAIPcenter}). Toward that end, we also show that for any $a\in C(Q)$, $a^2$ is a Moufang element (Theorem \ref{MoufangElement}). This immediately gives us that for each $a\in C(Q)$, $a^6\in Z(Q)$, where $Z(Q)$ denotes the center of $Q$ (Corollary \ref{SAIPcelement}). This simultaneously generalizes two results: that in a Moufang loop, the cube of any commutant element is central \cite{bruck}, and that in a commutative semiautomorphic IP loop, the sixth power of any element is central \cite{KV09}.
 
In {\S}\ref{sect4} we discuss two construction of semiautomorphic IP loops. There is a well-known doubling construction of Chein which builds nonassociative Moufang loops from nonabelian groups. The construction itself makes sense even when one starts with a loop instead of a group. It turns out that if one applies the construction to a semiautomorphic IP loop, the result is another semiautomorphic IP loop (Theorem \ref{Cconstruction}).  In particular, this allows us to construct nonMoufang, nonSteiner, semiautomorphic IP loops by starting with nonassociative Moufang loops.
 
We then give our second construction, which is based on another doubling technique of de Barros and Juriaans. It was already noted (without human proof) that applying the de Barros-Juriaans construction to a group gives what we are now calling a semiautomorphic IP loop. Here we show that just as with the Chein construction, starting with a semiautomorphic IP loop in the de Barros-Juriaans construction yields another semiautomorphic IP loop (Theorem \ref{dBJconstruction}).  In {\S}\ref{connections}, we consider connections between the two constructions.  Specifically, we show that if we start with a semiautomorphic IP loop, apply the de Barros-Juriaans construction and then apply the Chein construction to the result, we end up with the same loop up to isomorphism as if we had applied the Chein construction twice (Theorem \ref{twice}).
 
Finally in {\S}\ref{conclusion} we give conditions on when our constructions give commutative loops.  We also use our constructions to give some concrete examples of nonMoufang, nonSteiner, semiautomorphic IP loops.
\section{Semiautomorphic Inverse Property Loops}
\label{sect2}
Throughout juxtaposition binds more tightly than an explicit $\cdot$ so that, for instance, $xy\cdot z$ means $(xy)z$.  It is well known that the inner mapping group of any loop is generated by all inner mappings of the form $L_{x,y},R_{x,y},$ and $T_{x}$ \cite{bruck}, where
\begin{equation*}
T_{x}=R_{x}L_{x}^{-1} \qquad L_{x,y} = L_{x}L_{y}L_{yx}^{-1} \qquad R_{x,y}=R_{x}R_{y}R_{xy}^{-1}.
\end{equation*}
\begin{lemma}\cite{ikuta,KKP}.  Let $Q$ be an IP loop.  Then the following are equivalent:
\begin{itemize}
\item [$(\ref{originaldef}.1)$] For all $\theta \in Inn(Q)$, $x^{-1}\theta=(x\theta)^{-1}$.\label{RIF.1}
\item [$(\ref{originaldef}.2)$] $Q$ is flexible and $R_{x,y} = L_{x^{-1},y^{-1}}$ for all $x,y \in Q$.\label{RIF.2}
\item [$(\ref{originaldef}.3)$] $R_{xy}L_{xy} = L_{y}L_{x}R_{x}R_{y}$ for all $x,y \in Q$. \label{RIF.3}
\item [$(\ref{originaldef}.4)$] $L_{xy}R_{xy} = R_{x}R_{y}L_{y}L_{x}$ for all $x,y \in Q$. \label{RIF.4}
\end{itemize}
\label{originaldef}
\end{lemma}
By flexibility, the left hand sides of (\ref{RIF.3}.3) and (\ref{RIF.4}.4) are equal and thus we can equate (\ref{RIF.3}.3) with either side of (\ref{RIF.4}.4).  For convenience, define
\begin{equation*}
P_{x}=L_{x}R_{x}=R_{x}L_{x}
\end{equation*}
by flexibility.  Then in an IP loop conditions (\ref{RIF.3}.3) and (\ref{RIF.4}.4) can be written as
\begin{align}
\tag{RIF1} L_x P_y R_x &= P_{yx} \label{RIF1}, \\
\tag{RIF2} R_x P_y L_x &= P_{xy} \label{RIF2}.
\end{align}
We will use the RIF acronym as an equation label for historical reference.  We also use the ARIF condition,
\begin{equation}\tag{ARIF}
R_{x}R_{yxy}=R_{xyx}R_{y}, \qquad L_{x}L_{yxy}=L_{xyx}L_{y},
\label{ARIF}
\end{equation}
which hold in any loop satisfying the conditions of Lemma \ref{originaldef}; see \cite{KKP}.

\begin{theorem}
Let $Q$ be an IP loop satisfying \eqref{RIF1} and \eqref{RIF2}.  Then every inner mapping is a semiautomorphism.
\label{T1}
\end{theorem}
\begin{proof}  By (\ref{originaldef}.2), it is enough to show that each $T_{x}$ and each $R_{x,z}$ is a semiautomorphism.  Note that an inner mapping $\theta$ is a semiautomorphism \emph{if and only if} $P_{x}\theta=\theta P_{x\theta}$ for all $x\in Q$.  First, $1=1T_{x}=1R_{x,y}=1L_{x,y}$ by definition.  Thus we compute
\begin{equation*}
P_y T_x 
=\underbracket[0.75pt]{P_y R_x} L_{x^{-1}}
\stackrel{\eqref{RIF1}}{=}  L_{x^{-1}} \underbracket[0.75pt]{P_{yx} L_{x^{-1}}}
\stackrel{\eqref{RIF2}}{=}  L_{x^{-1}} R_x P_{x^{-1}yx}\\
= T_x P_{y T_x}.
\end{equation*}
For $R_{x,z}$, we compute
\begin{alignat*}{3}
P_y R_{x,z} 
&= \underbracket[0.75pt]{P_y R_x} R_z R_{(xz)^{-1}}
&&\stackrel{\eqref{RIF1}}{=} L_{x^{-1}} \underbracket[0.75pt]{P_{yx} R_z} R_{(xz)^{-1}}\\
&\stackrel{\eqref{RIF1}}{=} L_{x^{-1}} L_{z^{-1}} \underbracket[0.75pt]{P_{yx\cdot z} R_{(xz)^{-1}}}
&&\stackrel{\eqref{RIF1}}{=} L_{x^{-1}} L_{z^{-1}} L_{xz} P_{(yx\cdot z)(xz)^{-1}}\\
&= \underbracket[0.75pt]{L_{x^{-1}, z^{-1}}} P_{y R_{x,z}}
&&\stackrel{(\ref{originaldef}.2)}{=} R_{x,z} P_{y R_{x,z}}.
\end{alignat*}
\end{proof}
Hence, we have shown that semiautomorphic IP loops coincide with the variety formerly known as RIF loops.
\begin{proof}[Proof of Theorem \ref{samedef}.]
This follows immediately from Theorem \ref{T1} and the earlier observation that semiautomorphisms preserve inverses.
\end{proof} 
\section{Commutant of a Semiautomorphic loop}
\label{sect3}
Let $Q$ be a loop. Then we have the following subsets of interest.
\begin{itemize}
\item [$\bullet$] The \emph{commutant of Q}, 
\[C(Q)=\{a\in Q\mid ax=xa\ \forall x\in Q\}=\{a\in Q \mid L_{a}=R_{a}\}.\]
\item [$\bullet$] The \emph{nucleus of Q},
\[N(Q)=\{a\in Q\mid a\cdot xy=ax\cdot y,x\cdot ay=xa\cdot y,x\cdot ya=xy\cdot a,\ \forall x,y\in Q\}.\]
\item [$\bullet$] The \emph{center of Q}, $Z(Q)=C(Q)\cap N(Q)$.
\item [$\bullet$] The set of \emph{Moufang elements},
\[M(Q)=\{a\in Q\mid a(xy\cdot a)=ax\cdot ya, a(x\cdot ay)=(ax\cdot a)y, (ya\cdot x)a=y(a\cdot xa),\ \forall x,y\in Q\}.\]
\end{itemize}

It is well known that $Z(Q)$ and $N(Q)$ are always subgroups \cite{bruck}.  The set $M(Q)$ of Moufang elements is also a subloop of any loop \cite{JD}.  

In a Moufang loop $Q$, it is noted in \cite{bruck} that $C(Q)$ is a subloop and an explicit proof is given in \cite{hala}.  In this section we will prove the same result for semiautomorphic IP loops.

We note that in an IP loop $Q$, to verify that a subset $S$ is a subloop, it is sufficient to check that $S$ is closed under multiplication and taking inverses.  
\begin{theorem}
The commutant of a semiautomorphic IP loop is a subloop.
\label{center}
\end{theorem}
\noindent To this end, we will first prove the following.
\begin{theorem}
Let $Q$ be a semiautomorphic IP loop and let $a\in C(Q)$, then $a^{2}$ is a Moufang element.
\label{MoufangElement}
\end{theorem}

The proof will occupy most of this section and will require some technical lemmas.  We note that in a semiautomorphic IP loop $Q$, each $\theta \in \Inn(Q)$ preserves powers, that is, $x^{n}\theta=(x\theta)^{n}$ for all $x\in Q$, $n\in \mathbb{Z}$.  We will use this without comment in what follows.

\begin{lemma}\cite{KV09}.
In a semiautomorphic IP loop $Q$, $a\in M(Q)$ if and only if $(yx\cdot a)x=y\cdot xax$ for all $x,y\in Q$.
\label{othermoufelemdef}
\end{lemma}

\begin{lemma}
Let $Q$ be a dissociative loop and $a\in C(Q)$.  Then $\left\langle a \right\rangle \subseteq C(Q)$. 
\label{C1}
\end{lemma}
\begin{proof}
We simply note $a^{n} x = x L_{a^{n}} = x L_{a}^{n} = x R_{a}^{n} = x R_{a^{n}} = x a^n$.
\end{proof}
\begin{lemma}
Let $Q$ be a diassociative loop. For all $a\in C(Q)$, $x\in Q$ and all $n\in \mathbb{Z}$, $(xa)^{n}=x^{n}a^{n}$.
\label{C2}
\end{lemma}
\begin{proof}
This follows easily from Lemma \ref{C1} and an induction argument.
\end{proof}

\begin{lemma}
Let $Q$ be a semiautomorphic IP loop and let $a\in C(Q)$.  For all $x\in Q$
\begin{equation}\tag{\ref{L.1}.1}
P_{a}L_{x}R_{a}L^{-1}_{x}=L_{x}R_{a}L^{-1}_{x}P_{a}.
\label{LA}
\end{equation}
\label{L.1}
\end{lemma}
\begin{proof}
We have $P_a L_x R_a L_x^{-1} = P_a \theta R_a$ where $\theta = L_x R_a L_x^{-1} R_a^{-1}\in \Inn(Q)$.
Since $\theta$ is a semiautomorphism, $P_a \theta = \theta P_{a\theta}$. We have
$a\theta = a L_x R_a L_x^{-1} R_a^{-1} = a$ by diassociativity, and so
$P_a L_x R_a L_x^{-1} = L_x R_a L_x^{-1} R_a^{-1} P_a R_a = L_x R_a L_x^{-1} P_a$, as claimed.
\end{proof}
\begin{lemma}
Let $Q$ be a semiautomorphic IP loop and let $a\in C(Q)$.  For all $x\in Q$
\begin{equation}\tag{\ref{L.2}.1}
R^{2}_{ax}=R_{a}R^{2}_{x}R_{a} \qquad L^{2}_{ax}=L_{a}L^{2}_{x}L_{a}.
\label{LB}
\end{equation}
\label{L.2}
\end{lemma}
\begin{proof}
Using diassociativity, we have
\begin{equation*}
x
=(x\cdot ay)(ya)^{-1}
=(x\cdot ay)^{2}[(x\cdot ay)^{-1}(ya)^{-1}]\\
=(x\cdot ay)^{2}[ya\cdot x\cdot ay]^{-1},
\end{equation*}
and so
\begin{equation}\tag{\ref{L.2}.2}
x=(x\cdot ay)^{2}R_{(ya\cdot x\cdot ay)}^{-1}.
\label{LB1}
\end{equation}
Our intermediate goal is to prove
\begin{equation}\tag{\ref{L.2}.3}
[(xy)^2] L_{x^{-1}} L_a L_{ax} =  (a\cdot xy)^2
\label{LB2}
\end{equation}
We have
\[
[(xy)^2] L_{x^{-1}} L_a L_{ax} = [(xy)^2] \theta L_{a^2} = a^2 \cdot [(xy)^2]\theta
\]
where $\theta = L_{x^{-1}} L_a L_{ax} L_{a^{-2}} \in \Inn(Q)$. Since inner mappings preserve powers, we then have
\[
[(xy)^2] L_{x^{-1}} L_a L_{ax} = a^2 \cdot [(xy)\theta]^2  = [a\cdot (xy)\theta]^2 \,.
\]
Now
\begin{alignat*}{2}
a\cdot (xy)\theta 
&= (ax\cdot a(x^{-1}\cdot xy)) L_{a^{-2}} L_a 
&&= (ax) \underbracket[0.75pt]{R_{ay} P_{a^{-1}}} L_a \\
&\stackrel{\eqref{RIF1}}{=}  P_y L_{ay}^{-1} L_a 
&&= [\underbracket[0.75pt]{(ay)^{-1}}] R_{(ax)P_y} L_a \\
& = (x\cdot ay)R_{(ay\cdot x\cdot ay)^{-1}} R_{(ax)P_y} L_a 
&&= (x\cdot ay) \varphi \,,
\end{alignat*}
where $\varphi = R_{x P_{ay}}^{-1} R_{(ax)P_y} L_a$ and where the fifth equality follows from diassociativity.  Thus
\[
[(xy)^2] L_{x^{-1}} L_a L_{ax} = [(x\cdot ay)\varphi]^2\,.
\]
Since
\[
x P_{ay} \stackrel{\eqref{RIF1}}{=} (xa)P_y L_a = a\cdot (ax)P_y\,,
\]
we see that $\varphi = R_{a\cdot (ax)P_y}^{-1} R_{(ax)P_y} L_a$ is an inner mapping. So putting our calculations together, we have
\begin{alignat*}{2}
[(xy)^2] L_{x^{-1}} L_a L_{ax} 
&= [(x\cdot ay)^2]\varphi  
&&= \underbracket[0.75pt]{[(x\cdot ay)^2] R_{(ay\cdot x\cdot ay)^{-1}}} R_{(ax)P_y} L_a \\
&\stackrel{\eqref{LB1}}{=}  x R_{(ax)P_y} L_a  
&&= a \underbracket[0.75pt]{R_x P_y L_x} L_a \\
&\stackrel{\eqref{RIF1}}{=}  a P_{xy} L_a 
&&= a\cdot xy \cdot a \cdot xy \\
&= (a\cdot xy)^2\,.
\end{alignat*}
This establishes \eqref{LB2}.  Now in \eqref{LB2}, replace $y$ with $x^{-1} y$ and rearrange to get
\[
(ax)^{-1} \cdot (ay)^2 = a\cdot x^{-1} y^2\,.
\]
Then replace $x$ with $(ax)^{-1}$ and simplify to get
\[
x\cdot (ay)^2 = a (ax\cdot y^2 )\,,
\]
that is, $x R_{(ay)^2} = x R_a R_{y^2} R_a$. This establishes half of the desired result, and the other half follows by a dual argument.
\end{proof}

\begin{lemma}
Let $Q$ be a semiautomorphic IP loop and let $a\in C(Q)$.  For all $x\in Q$
\begin{equation}\tag{\ref{L.3}.1}
R_{x}R^{-1}_{a^{2}x}R_{a}=R_{a}R_{x}R^{-1}_{a^{2}x}.
\label{LC}
\end{equation}
\label{L.3}
\end{lemma}
\begin{proof}
Since $x = ax\cdot (a^2x)^{-1}\cdot ax$, we have
\[
R_x R_{a^2 x}^{-1} R_a = \underbracket[0.75pt]{R_{ax\cdot (a^2x)^{-1}\cdot ax} R_{(a^2 x)^{-1}}} R_a
\stackrel{\eqref{ARIF}}{=}  R_{ax} R_{(a^2x)^{-1}\cdot ax\cdot (a^2x)^{-1}} R_a \,.
\]
Now $(a^2x)^{-1}\cdot ax\cdot (a^2x)^{-1} = a^{-1}\cdot (ax)^{-1}\cdot a^{-1}$ and $R_{ax} = R_{ax}^2 R_{(ax)^{-1}}$, and so
by the above,
\[
R_x R_{a^2 x}^{-1} R_a =R_{ax}^2 \underbracket[0.75pt]{R_{(ax)^{-1}} R_{a^{-1}\cdot (ax)^{-1}\cdot a^{-1}}} R_a
\stackrel{\eqref{ARIF}}{=}  R_{ax}^2  R_{(ax)^{-1}\cdot a^{-1}\cdot (ax)^{-1}} R_{a^{-1}} R_a \,.
\]
We have $R_{ax}^2 = R_a R_{x^2} R_a$ by Lemma \ref{L.2} and $(ax)^{-1}\cdot a^{-1}\cdot (ax)^{-1} = (a^2x)^{-1}\cdot a\cdot (a^2x)^{-1}$, and so
\begin{alignat*}{2}
R_x R_{a^2 x}^{-1} R_a 
&= R_a R_{x^2} \underbracket[0.75pt]{R_a R_{ (a^2x)^{-1}\cdot a\cdot (a^2x)^{-1}}}  
&&\stackrel{\eqref{ARIF}}{=} R_a R_{x^2} R_{a\cdot (a^2x)^{-1}\cdot a} R_{(a^2x)^{-1}}  \\
&= R_a R_{x^2} R_{x^{-1}} R_{(a^2x)}^{-1} 
&&= R_a R_x R_{(a^2x)}^{-1} \,,
\end{alignat*}
as claimed.
\end{proof}
\begin{lemma}
Let $Q$ be a semiautomorphic IP loop and let $a\in C(Q)$.  For all $x\in Q$
\begin{equation}\tag{\ref{L.4}.1}
(xy)^{-1}\cdot ax=ax\cdot (yx)^{-1}.
\label{LD}
\end{equation}
\label{L.4}
\end{lemma}
\begin{proof}
By diassociativity, $ax = (xy\cdot y^{-1})a = (xy\cdot (a\cdot (a^2 y)^{-1}\cdot a)) a$, and so we have
\begin{alignat*}{3}
(xy)^{-1}\cdot ax
&=(xy)^{-1}\cdot (xy\cdot (a\cdot (a^{2}y)^{-1}\cdot a))a
&&=[(a^{2}y)^{-1}]\underbracket[0.75pt]{P_{a}L_{xy}R_{a}L_{(xy)^{-1}}}\\
&\stackrel{\eqref{LA}}{=}[(a^{2}y)^{-1}]L_{xy}R_{a}L_{(xy)^{-1}}P_{a}
&&=x\underbracket[0.75pt]{R_{y}R_{(a^{2}y)^{-1}}R_{a}}L_{(xy)^{-1}}P_{a}\\
&\stackrel{\eqref{LC}}{=}xR_{a}R_{y}R_{(a^{2}y)^{-1}}L_{(xy)^{-1}}R_{a}.
\end{alignat*}
Now
\begin{alignat*}{3}
(xa\cdot y)(a^2 y)^{-1} 
&= (a^2 y)^{-1} \underbracket[0.75pt]{P_{xa\cdot y}} R_{(xa\cdot y)^{-1}} 
&&\stackrel{\eqref{RIF2}}{=}  \underbracket[0.75pt]{(a^2 y)^{-1} L_y P_{xa}} R_y R_{(xa\cdot y)^{-1}} \\
&= (x^2) R_y  R_{(xa\cdot y)^{-1}} 
&&= (x^2 y) R_{(xa\cdot y)^{-1}}\,,
\end{alignat*}
using diassociativity in the third equality. Combining this with the calculation above, we have
\[
(xy)^{-1} \cdot ax = \underbracket[0.75pt]{(x^2y) R_{(xa\cdot y)^{-1}}} L_{(xy)^{-1}} P_a
= [(xa\cdot y)^{-1}] L_{x^2 y} L_{(xy)^{-1}} P_a \,.
\]
Since $x^2 y = xy\cdot y^{-1}\cdot xy$, we get
\begin{alignat*}{3}
(xy)^{-1} \cdot ax 
&= [(xa\cdot y)^{-1}] L_y \underbracket[0.75pt]{L_{y^{-1}} L_{ xy\cdot y^{-1}\cdot xy }} L_{(xy)^{-1}} P_a\\
&\stackrel{\eqref{ARIF}}{=} [(xa\cdot y)^{-1}] L_y L_{y^{-1}\cdot xy\cdot y^{-1}} \underbracket[0.75pt]{L_{xy} L_{(xy)^{-1}}}P_a
&&= [(xa\cdot y)^{-1}] L_y L_{y^{-1}\cdot xy\cdot y^{-1}} P_a.
\end{alignat*}
Now $[(xa\cdot y)^{-1}] L_y = (xa)^{-1}$ and $y^{-1}\cdot xy\cdot y^{-1} = y^{-1} x$, and so
\begin{alignat*}{3}
(xy)^{-1} \cdot ax 
&= [(xa)^{-1}] L_{y^{-1}x} P_a 
&&= (y^{-1}x) \underbracket[0.75pt]{R_{(xa)^{-1}} P_a L_{(ax)^{-1}}} L_{ax}  \\
&\stackrel{\eqref{RIF1}}{=} \underbracket[0.75pt]{(y^{-1} x) P_{x^{-1}}} L_{ax} 
&&= (x^{-1}y^{-1}) L_{ax} \\
&= ax\cdot (yx)^{-1},
\end{alignat*}
using $(xa)^{-1} \cdot a = x^{-1}$ in the second equality. This completes the proof.
\end{proof}
\begin{lemma}
Let $Q$ be a semiautomorphic IP loop and let $a\in C(Q)$.  For all $x\in Q$
\begin{equation}\tag{\ref{L.5}.1}
T_{ax}=T_{x}
\label{LE}
\end{equation}
\label{L.5}
\end{lemma}
\begin{proof}
Invert both sides of \eqref{LD} to get
\begin{equation*}
(ax)^{-1}\cdot xy=yx\cdot (ax)^{-1}
\end{equation*}
which is
\begin{equation*}
L_{x}L^{-1}_{ax}=R_{x}R^{-1}_{ax}.
\end{equation*}
Rearranging gives
\begin{equation*}
L^{-1}_{ax}R_{ax}=L^{-1}_{x}R_{x},
\end{equation*}
which establishes the claim.
\end{proof}
\noindent We are now ready to prove the two main results of this section.
\begin{proof}[Proof of Theorem \ref{MoufangElement}] 
Let $a\in C(Q)$.  Then for all $x\in C(Q)$,
\begin{alignat*}{4}
R_x R_{a^2} R_x 
&= R_x \underbracket[0.75pt]{P_a R_x}
&&\stackrel{\eqref{RIF1}}{=} R_x L_x^{-1} P_{ax}
&&&&= \underbracket[0.75pt]{T_x}  P_{ax}\\
&= T_{ax} P_{ax}
&&= R_{ax} L_{ax}^{-1} R_{ax} L_{ax}
&&&&= R_{ax} L_{ax}^{-1} L_{ax} R_{ax}\\
&= R_{ax}^2
&&= R_{(ax)^2}
&&&&= R_{x a^2 x},
\end{alignat*}
where the fourth equality follows from Lemma \ref{L.5}.  Hence, $(yx\cdot a^2)x=yR_{x}R_{a^{2}}R_{x}=yR_{xa^{2}x}= y (x\cdot  a^2\cdot x)$.  By Lemma \ref{othermoufelemdef}, we have the desired result.
\end{proof}
\begin{proof}[Proof of Theorem \ref{center}]
Let $a,b\in C(Q)$.  Then, for all $x,y\in Q$,
\begin{alignat*}{4}
ab\cdot x\cdot ab 
&\stackrel{\eqref{RIF2}}{=} a(b\cdot xa\cdot b)
&&= (xa\cdot b)b\cdot a  
&&&= \underbracket[0.75pt]{(xa\cdot b^2)\cdot a} \\
&= x\cdot ab^2 a
&&= x\cdot (ab)^2 
&&&= (x\cdot ab)\cdot ab,
\end{alignat*}
where the fourth equality follows from the fact that $b^{2}$ is a Moufang element, Theorem \ref{MoufangElement}.  Hence, cancelling $ab$ on the right gives $ab\cdot x=x\cdot ab$.
\end{proof}

\begin{lemma}\cite{bruck}.  Let $Q$ be an IP loop.  Then for every $x\in M(Q)\cap C(Q)$, $x^{3} \in Z(Q)$.
\label{MelemCenter}
\end{lemma}
\noindent
Thus, we have the following,
\begin{corollary}
Let $Q$ be a semiautomorphic IP loop.  If $a\in C(Q)$, then $a^{6} \in Z(Q)$.
\label{SAIPcenter}
\end{corollary}
\begin{proof}
This immediately follows from Theorem \ref{MoufangElement} and Lemma \ref{MelemCenter}.
\end{proof}
An element $a$ of a loop $Q$ is a \emph{C-element} if it satisfies the following equation for all $x,y\in Q$.
\begin{equation}\tag{$C_{0}$}
x(a\cdot ay)=(xa\cdot a)y
\end{equation}
We denote $C_{0}(Q)$ be the set of all \emph{c-elements} in a loop $Q$.
\begin{lemma}\cite{chein3}.  In an IP loop $Q$, $a\in C_{0}(Q)$ if and only if $a^{2}\in N(Q)$.
\label{celement}
\end{lemma}
\noindent Hence we have the following,
\begin{corollary}
Let $Q$ be a semiautomorphic IP loop.  If $a\in C(Q)$, then $a^{3}$ is a $c$-element.
\label{SAIPcelement}
\end{corollary}
\begin{proof}
This follows immediately from Theorem \ref{SAIPcenter} and Lemma \ref{celement}.
\end{proof}
Recall that semiautomorphic IP loops are generalized by flexible loops satisfying \eqref{ARIF} \cite{KKP}.  The following shows we cannot generalize Theorems \ref{T1} and \ref{center} to such loops.
\begin{example}
Let $(Q,\cdot)$ be a loop with multiplication given by Table \ref{multtable}.  Then $Q$ is a flexible, nonsemiautomorphic, IP, C-loop of order $20$ were $R_{x,y}$ and $L_{x,y}$ are not semiautomorphisms and the commutant is not a subloop, found by \textsc{Mace4} \cite{PM}.  
\begin{table}[ht]
\footnotesize
\begin{tabular}{c|cccccccccccccccccccc} 
$\cdot$& 1 & 2 & 3 & 4 & 5 & 6 & 7 & 8 & 9 & 10 & 11 & 12 & 13 & 14 & 15 & 16 & 17 & 18 & 19 & 20 \\ \hline
1& 1 & 2 & 3 & 4 & 5 & 6 & 7 & 8 & 9 & 10 & 11 & 12 & 13 & 14 & 15 & 16 & 17 & 18 & 19 & 20 \\ 
2&2 & 4 & 1 & 3 & 7 & 8 & 6 & 5 & 10 & 12 & 9 & 11 & 19 & 20 & 17 & 18 & 16 & 15 & 14 & 13 \\ 
3&3 & 1 & 4 & 2 & 8 & 7 & 5 & 6 & 11 & 9 & 12 & 10 & 20 & 19 & 18 & 17 & 15 & 16 & 13 & 14 \\ 
4&4 & 3 & 2 & 1 & 6 & 5 & 8 & 7 & 12 & 11 & 10 & 9 & 14 & 13 & 16 & 15 & 18 & 17 & 20 & 19 \\ 
5&5 & 7 & 8 & 6 & 4 & 1 & 3 & 2 & 13 & 17 & 18 & 14 & 12 & 9 & 19 & 20 & 11 & 10 & 16 & 15 \\ 
6&6 & 8 & 7 & 5 & 1 & 4 & 2 & 3 & 14 & 18 & 17 & 13 & 9 & 12 & 20 & 19 & 10 & 11 & 15 & 16 \\ 
7&7 & 6 & 5 & 8 & 3 & 2 & 1 & 4 & 15 & 19 & 20 & 16 & 17 & 18 & 9 & 12 & 13 & 14 & 10 & 11 \\ 
8&8 & 5 & 6 & 7 & 2 & 3 & 4 & 1 & 16 & 20 & 19 & 15 & 18 & 17 & 12 & 9 & 14 & 13 & 11 & 10 \\ 
9&9 & 10 & 11 & 12 & 13 & 14 & 16 & 15 & 1 & 2 & 3 & 4 & 5 & 6 & 8 & 7 & 19 & 20 & 17 & 18 \\ 
10&10 & 12 & 9 & 11 & 17 & 18 & 20 & 19 & 2 & 4 & 1 & 3 & 16 & 15 & 13 & 14 & 6 & 5 & 7 & 8 \\ 
11&11 & 9 & 12 & 10 & 18 & 17 & 19 & 20 & 3 & 1 & 4 & 2 & 15 & 16 & 14 & 13 & 5 & 6 & 8 & 7 \\ 
12&12 & 11 & 10 & 9 & 14 & 13 & 15 & 16 & 4 & 3 & 2 & 1 & 6 & 5 & 7 & 8 & 20 & 19 & 18 & 17 \\ 
13&13 & 19 & 20 & 14 & 12 & 9 & 18 & 17 & 5 & 15 & 16 & 6 & 4 & 1 & 11 & 10 & 7 & 8 & 3 & 2 \\ 
14&14 & 20 & 19 & 13 & 9 & 12 & 17 & 18 & 6 & 16 & 15 & 5 & 1 & 4 & 10 & 11 & 8 & 7 & 2 & 3 \\ 
15&15 & 17 & 18 & 16 & 19 & 20 & 12 & 9 & 7 & 14 & 13 & 8 & 10 & 11 & 4 & 1 & 3 & 2 & 6 & 5 \\ 
16&16 & 18 & 17 & 15 & 20 & 19 & 9 & 12 & 8 & 13 & 14 & 7 & 11 & 10 & 1 & 4 & 2 & 3 & 5 & 6 \\ 
17&17 & 16 & 15 & 18 & 11 & 10 & 14 & 13 & 19 & 6 & 5 & 20 & 8 & 7 & 3 & 2 & 1 & 4 & 9 & 12 \\ 
18&18 & 15 & 16 & 17 & 10 & 11 & 13 & 14 & 20 & 5 & 6 & 19 & 7 & 8 & 2 & 3 & 4 & 1 & 12 & 9 \\ 
19&19 & 14 & 13 & 20 & 16 & 15 & 11 & 10 & 17 & 8 & 7 & 18 & 3 & 2 & 6 & 5 & 9 & 12 & 1 & 4 \\ 
20&20 & 13 & 14 & 19 & 15 & 16 & 10 & 11 & 18 & 7 & 8 & 17 & 2 & 3 & 5 & 6 & 12 & 9 & 4 & 1 \\ 
 \end{tabular}
\caption{Multiplication Table for $(Q,\cdot)$}
\label{multtable}
\end{table}
\end{example}
\normalsize

\section{Constructing semiautomorphic IP loops}
\label{sect4}
We now give two constructions of semiautomorphic IP loops.  We follow the notation given in \cite{chein1, chein2, KPV}.  To show that $Q$ is a semiautomorphic IP loop, by Theorem \ref{samedef}, it is enough to show $Q$ is an IP loop and satisfies either \eqref{RIF1} or \eqref{RIF2}.  The following will be used without comment.
\begin{lemma}\cite{KPV}.
Let $Q$ be an IP loop and $*:Q\rightarrow Q$ a bijection such that $gg^{*} \in Z(Q)$ for every $g \in Q$.  Then $g^{*}g=gg^{*} \in Z(Q)$ for every $g \in Q$.
\label{L1}
\end{lemma}
\begin{proof} Since $Q$ in an IP loop,
\begin{equation*}
g^{*}g=(g^{-1} \cdot gg^{*})g=(gg^{*} \cdot g^{-1})g=gg^{*}. \qedhere
\end{equation*}
\end{proof}

Our first construction for semiautomorphic IP loops is based on Chein's $M(G,*,g_0)$ Moufang loop from nonabelian groups \cite{chein1, chein2}.  We begin with the following lemma.

\begin{lemma}
Let $Q$ be a semiautomorphic IP loop and let $*$ be a semiautomorphism of $Q$ such that
\begin{align}
\tag{\ref{semiL1}.1} (g^{*})^{*}&=g, \label{sl.1} \\
\tag{\ref{semiL1}.2}  g^{*}h \cdot (k \cdot g^{*}h)^{*} &= (g \cdot h^{*}k)^{*}g \cdot h^{*}. \label{sl.2}
\end{align}
Then for all $g,h\in Q$,
\begin{align}
\tag{\ref{semiL1}.3} g(hg)^{*} &= (g^{*}h)^{*}g^{*}, \label{sl.2.i} \\
\tag{\ref{semiL1}.4} ((gh)^* g)^* &= (g^* h^*)^* g^*, \label{sl.2.ii} \\
\tag{\ref{semiL1}.5} (g(hg)^{*})^{*} &= g^{*}(h^{*}g^{*})^{*}. \label{sl.2.iii}
\end{align}
\label{semiL1}
\end{lemma}
\begin{proof}
Recall that $(x^{-1})^{*}=(x^{*})^{-1}$ for all $x\in Q$ since $*$ is a semiautomorphism.  For \eqref{sl.2.i}, simply let $g=1$ in \eqref{sl.2}.  For \eqref{sl.2.ii}, we see 
\begin{equation*}
g^{*} = ((g^{*}h)^{-1})^{*}((g^{*}h)^{*} \cdot g^{*}) = ((g^{*}h)^{-1})^{*}(g(hg)^{*})=(h^{-1}(g^{-1})^{*})^{*}(g(hg)^{*}).
\end{equation*}
Replace $h$ with $h^{-1}$ and then interchange $g$ and $h$ gives $h=(gh^{-1})^{*}(h^{*}(g^{-1}h^{*})^{*})$.  Applying this to \eqref{sl.2}, we get
\begin{alignat*}{3}
(hg)^{*}
&= (h(gh^{-1}\cdot h))^{*}
&&=(\underbracket[0.75pt]{h^{*}} \cdot(gh^{-1})^{*})\cdot h^{*}\\
&=\underbracket[0.75pt]{\left((g^{-1}h)^{*} \left[(h^{*}(g^{-1}h^{*}))^{*}\cdot (gh^{-1})^{*}\right]\right)^{*}h^{*}}
&&\stackrel{\eqref{sl.2}}{=}\left[((gh^{-1})^{*})^{*}h\right]\left[(g^{-1}h^{*})^{*}(((gh^{-1})^{*})^{*}h)\right]^{*}\\
&=\left[gh^{-1} \cdot h\right]\left[(g^{-1}h^{*})^{*}(gh^{-1}\cdot h)\right]^{*}
&&=g\left[(g^{-1}h^{*})^{*}g\right]^{*}.
\end{alignat*}
Replacing $h$ with $(g^*h)^*$, $g$ with $g^*$, and using \eqref{sl.2.i}, we have
\begin{alignat*}{3}
(g(hg)^{*})^{*}
&\stackrel{\eqref{sl.2.i}}{=}\underbracket[0.75pt]{((g^{*}h)^{*}g^{*})^{*}}
&&=g^{*}\left(\left[(g^{-1})^{*} ((g^{*}h)^{*})^{*}\right]^{*}g^{*}\right)^{*}\\
&=g^{*}\left(\left[(g^{-1})^{*}(g^{*}h)\right]^{*}g^{*}\right)^{*}
&&=\underbracket[0.75pt]{g^{*}(h^{*}g^{*})^{*}}\\
&\stackrel{\eqref{sl.2.i}}{=}(gh^{*})^{*}g.
\end{alignat*}
Therefore, we have $(g^{*}h^{*})^{*}g^{*}= (g^{*}(hg^{*})^{*})^{*} = ((gh)^{*}g)^{*}$.  Lastly, \eqref{sl.2.iii} follows from \eqref{sl.2.ii} and the previously stated fact that semiautomorphisms respect inverses.
\end{proof}

\begin{lemma}  
Let Q be a semiautomorphic IP loop, let $g_0 \in Z(Q)$ be fixed and let $*$ be a semiautomorphism of $Q$ such that, for all $g,h,k \in Q$
\begin{align}
\tag{\ref{L3}.1} (g^{*})^{*}&=g, \label{L3.1} \\
\tag{\ref{L3}.2} (gg_{0})^{*}&=g^{*}g_{0}, \label{L3.2} \\
\tag{\ref{L3}.3}  g^{*}h \cdot (k \cdot g^{*}h)^{*} &= (g \cdot h^{*}k)^{*}g \cdot h^{*}. \label{L3.3}
\end{align}
For an indeterminate $t$, define multiplication $\circ$ on $Q \cup Qt$ by
\begin{equation*}
g\circ h =gh, \qquad g \circ(ht) = (g^{*}h^{*})^{*}t, \qquad gt \circ h = (gh^{*})t, \qquad gt \circ ht = g_{0}(g^{*}h)^{*},
\end{equation*}
where $g,h \in Q$.  Then $(Q \cup Qt,\circ)$ is a semiautomorphic IP loop.
\label{L3}
\end{lemma}
\begin{proof} 
Let $x,y,z \in (Q\cup Qt,\circ)$.  The calculations for $\circ$ are straightforward and left to the reader.  However, the following eight equalities must be verified.  Note that we have moved $g_0$ to the far left in each expression.
\begin{center}
\begin{tabular}{|l|l|l|l|l|l|}\hline
{\bf Cases}& $\bf x$&$\bf y$&$\bf z$&$\bf (x\circ y) \circ (z \circ (x\circ y))$&$\bf ((x \circ (y\circ z))\circ x) \circ y$\\\hline
{\bf Case 1:}& $g$&$h$&$k$&$gh\cdot (k\cdot gh)$&$(g \cdot hk)g \cdot h$\\\hline
{\bf Case 2:}& $g$&$h$&$kt$&${[(gh)^{*} (k \cdot (gh)^{*})^{*}]^{*}}t$&$[(g^{*} \cdot h^{*}k^{*})^{*}g^{*}\cdot h^{*}]t$\\\hline
{\bf Case 3:}& $g$&$ht$&$k$&$g_{0}[g^{*}h^{*} \cdot (k^{*} \cdot g^{*}h^{*})^{*}]^{*}$&$g_{0}[(g^{*}(hk^{*})^{*}\cdot g^{*})^{*}\cdot h]^{*}$\\\hline
{\bf Case 4:}& $gt$&$h$&$k$&$g_{0}[(gh^{*})^{*} \cdot (k^{*} \cdot (gh^{*})^{*})^{*}]^{*}$&$g_{0}[((g \cdot (hk)^{*})^{*}\cdot g)^{*}h]$\\\hline
{\bf Case 5:}& $gt$&$ht$&$k$&$g_{0}g_{0}(g^{*}h \cdot (k^{*} \cdot (g^{*}h)))^{*}$&$g_{0}g_{0}((g^{*} \cdot hk^{*})g^{*} \cdot h)^{*}$\\\hline
{\bf Case 6:}& $gt$&$h$&$kt$&$g_{0}((gh^{*}) \cdot (k^{*}\cdot gh^{*}))t$&$g_{0}(((g^{*} \cdot (h^{*}k^{*})^{*})g^{*})^{*} \cdot h^{*})t$\\\hline
{\bf Case 7:}& $g$&$ht$&$kt$&$g_{0}((g^{*}h^{*})^{*} \cdot (k^{*} \cdot (g^{*}h^{*})^{*}))t$&$g_{0}((g(h^{*}k)^{*} \cdot g)^{*}\cdot h^{*})^{*}t$\\\hline
{\bf Case 8:}& $gt$&$ht$&$kt$&$g_{0}g_{0}((g^{*}h)^{*} \cdot (k\cdot (g^{*}h)^{*})^{*})^{*}t$&$g_{0}g_{0}((g \cdot h^{*}k)^{*}g \cdot h^{*})^{*}t$\\\hline
\end{tabular}
\end{center}
Note that cases $1$ and $5$ follow directly from \eqref{RIF1}.  Similarly, case $8$ follows from \eqref{L3.3}.  

For the case 2, we have
\[\underbracket[0.75pt]{[(gh)^{*} (k \cdot (gh)^{*})^{*}]^{*}}t
\stackrel{\eqref{sl.2.i}}{=}[\underbracket[0.75pt]{gh \cdot (k^{*} \cdot gh)^{*}}]t
\stackrel{\eqref{L3.3}}{=}[(g^{*} \cdot h^{*}k^{*})^{*}g^{*}\cdot h^{*}]t .\]

For case 3, we have
\[g_{0}[\underbracket[0.75pt]{g^{*}h^{*} \cdot (k^{*} \cdot g^{*}h^{*})^{*}}]^{*}
\stackrel{\eqref{L3.3}}{=}g_{0}\underbracket[0.75pt]{[(g \cdot hk^{*})^{*}g^{*}} \cdot h]^{*}
\stackrel{\eqref{sl.2.ii}}{=}g_{0}[(g^{*}(hk^{*})^{*}\cdot g^{*})^{*}\cdot h]^{*}.\]

For case 4, we have
\[g_{0}\underbracket[0.75pt]{[(gh^{*})^{*} \cdot (k^{*} \cdot (gh^{*})^{*})^{*}]^{*}}
\stackrel{\eqref{sl.2.iii}}{=}g_{0}[\underbracket[0.75pt]{gh^{*} \cdot (k \cdot gh^{*})^{*}}]
\stackrel{\eqref{L3.3}}{=}g_{0}[((g \cdot (hk)^{*})^{*}\cdot g)^{*}h].\]

For case 6, we have
\[g_{0}(\underbracket[0.75pt]{(gh^{*}) \cdot (k^{*}\cdot gh^{*})})t 
\stackrel{\eqref{RIF1}}{=}g_{0}(\underbracket[.075pt]{(g \cdot h^{*}k^{*})g} \cdot h^{*})t
\stackrel{\eqref{sl.2.i}}{=}g_{0}(((g^{*} \cdot (h^{*}k^{*})^{*})g^{*})^{*} \cdot h^{*})t.\]

For case 7, we have
\begin{align*}
g_{0}(\underbracket[0.75pt]{(g^{*}h^{*})^{*} \cdot (k^{*} \cdot (g^{*}h^{*})^{*})})t
&\stackrel{\eqref{sl.2.i}}{=}g_{0}(\underbracket[0.75pt]{(g^{*}h^{*}) \cdot (k \cdot (g^{*}h^{*}))})^{*}t
\stackrel{\eqref{RIF1}}{=}g_{0}\underbracket[0.75pt]{((g^{*} \cdot h^{*}k)g^{*}\cdot h^{*})^{*}}t\\
&\stackrel{\eqref{sl.2.i}}{=}g_{0}((g(h^{*}k)^{*} \cdot g)^{*}\cdot h^{*})^{*}t.
\end{align*}

Now, to see $(Q\cup Qt,\circ)$ is an IP loop, suppose $x\in Qt$ with $x=gt$ for some $g\in Q$.  Then note
\begin{equation*}
1\circ x=1\circ gt=(1g^{*})^{*}t=gt=x=gt=(g1^{*})t=gt\circ 1=x\circ 1.
\end{equation*}
Moreover, $x^{-1}=(gt)^{-1}=(g_{0}^{-1}g^{-*})t$, where $g^{-*}=(g^{-1})^{*}=(g^{*})^{-1}$.  For $x^{-1}\circ (x\circ y)=y$, we have the following $4$ cases:\\
\textbf{Case 1.}  Let $x=g,y=h$ for some $g,h\in Q$.  Thus
\[
g^{-1}\circ (g\circ h) = g^{-1}(gh)=h.
\]
\textbf{Case 2.}  Let $x=g,y=ht$ for some $g,h\in Q$.  Thus
\[
g^{-1}\circ (g\circ ht)=g^{-1}\circ (g^{*}h)^{*}t=(g^{-*}(g^{*}h^{*}))^{*}t=(h^{*})^{*}t=ht.
\]
\textbf{Case 3.}  Let $x=gt,y=h$ for some $g,h\in Q$.  Thus
\[
(gt)^{-1}\circ (gt\circ h)=(g_{0}^{-1}g^{-*})t\circ (gh^{*})t=g_{0}[g_{0}^{-1}(g^{-*})^{*}\cdot gh^{*}]^{*}=[g^{-1}\cdot gh^{*}]^{*}=(h^{*})^{*}=h.
\]
\textbf{Case 4.}  Let $x=gt,y=ht$ for some $g,h\in Q$.  Thus
\[
(gt)^{-1}\circ (gt\circ ht)=(g_{0}^{-1}g^{-*})t\circ(g_{0}(g^{*}h)^{*})=[g_{0}^{-1}g^{-*}\cdot (g_{0}(g^{*}h)^{*})^{*}]t=[g^{-*}\cdot g^{*}h]t=ht.
\]
Finally, $(y\circ x)\circ x^{-1}$ follows by a similar argument and is left to the reader.
\end{proof}
\begin{theorem}  Let $Q$ be a semiautomorphic IP loop, $g_{0}\in Z(Q)$, and $*$ an involutory antiautomorphism of $G$ such that $g_{0}^{*} = g_{0}, gg^{*} \in Z(Q)$ for every $g \in Q$.  For an indeterminate $t$, define multiplication $\circ$ on $Q \cup Qt$ by
\begin{equation*}
g\circ h =gh, \qquad g \circ(ht) = (hg)t, \qquad gt \circ h = (gh^{*})t, \qquad gt \circ ht = g_{0}h^{*}g,
\end{equation*}
where $g,h \in Q$.  Then $(Q \cup Qt,\circ)$ is a semiautomorphic IP loop.
\label{Cconstruction}
\end{theorem}
\begin{proof}
We see that by letting $*$ be an involutory antiautomorphism, \eqref{L3.1}, \eqref{L3.2} and \eqref{L3.3} of Lemma \ref{L3} are satisfied.  Note that multiplication in Lemma \ref{L3} becomes the multiplication in Theorem \ref{Cconstruction}.
\end{proof}
We now move to our second construction, based on de Barros and Juriaans' construction \cite{BJ1,BJ2}.  We note that if $Q$ is commutative, then multiplication defined by Theorem \ref{Cconstruction} is equivalent to multiplication from Theorem \ref{dBJconstruction}.  We begin with the following lemma.

\begin{lemma}
Let $Q$ be an IP Loop and $*$ be a bijection such that $gg^{*}\in Z(Q)$ for all $g\in Q$.  Then $\forall g,h \in Q$,
\begin{align}
\tag{\ref{TM}.1} g \cdot hh^{*} &= gh \cdot h^{*}, \label{A1}\\
\tag{\ref{TM}.2} gg^{*} \cdot h &= g\cdot g^{*}h, \label{A2}\\
\tag{\ref{TM}.3} g(hh^{*})\cdot g^{*} &= gh \cdot h^{*}g^{*}. \label{A3}
\end{align}
\label{TM}
\end{lemma}
\begin{proof}
For \eqref{A1}, simply note that
\begin{equation*}
g \cdot hh^{*} 
= (gh \cdot h^{-1}) \cdot hh^{*}
= gh \cdot (h^{-1} \cdot hh^{*})
= gh \cdot h^{*}.
\end{equation*}
Similarly for \eqref{A2}, since $gg^{*} \in Z(Q)$, we have $g^{-1}(gg^{*}\cdot h) = (g^{-1} \cdot gg^{*})h=g^{*}h$.  Multiply by $g$ on the left to get $gg^{*}\cdot h = g \cdot g^{*}h$.

For \eqref{A3}, we see $hh^{*}=hh^{*}\cdot g^{-1}g=g^{-1}(hh^{*} \cdot g)\stackrel{\eqref{A2}}{=}g^{-1}\cdot h(h^{*}g)$.  Now replace $g$ with $(h^{*})^{-1}g$ to derive $g^{-1}h^{*} \cdot hg = hh^{*}$.  Now,
\begin{equation*}
g^{*}h \cdot k 
= (g^{-1}\cdot gg^{*})h \cdot k
=gg^{*} (g^{-1}h \cdot k)
\stackrel{\eqref{A2}}{=}g \cdot g^{*}(g^{-1}h \cdot k).
\end{equation*}
Finally, substitute $k=(hg)$ and $h=h^{*}$ so that
\begin{equation*}
g^{*}h^{*} \cdot hg 
= g \cdot g^{*}(\underbracket[0.75pt]{g^{-1}h^{*} \cdot hg})
=gg^{*}\cdot (hh^{*})
\stackrel{\eqref{A2}}{=}h(gg^{*})\cdot h^{*}.
\end{equation*}
Since $g^{*}h^{*} \cdot hg = hg \cdot g^{*}h^{*}$, we have $hg \cdot g^{*}h^{*}=h(gg^{*})\cdot h^{*}$.
\end{proof}

\begin{theorem}  Let $Q$ be a semiautomorphic IP loop, $g_{0}\in Z(Q)$, and $*$ an involutory antiautomorphism of $Q$ such that $g_{0}^{*} = g_{0}, gg^{*} \in Z(Q)$ for every $g \in Q$.  For an indeterminate $t$, define multiplication $\circ$ on $Q \cup Qt$ by
\begin{equation*}
g\circ h =gh, \qquad g \circ(ht) = (gh)t, \qquad gt \circ h = (h^{*}g)t, \qquad gt \circ ht = g_{0}gh^{*},
\end{equation*}
where $g,h \in Q$.  Then $(Q \cup Qt,\circ)$ is a semiautomorphic IP loop.
\label{dBJconstruction}
\end{theorem}

\begin{proof}
As before, we summarize the eight cases below.
\begin{center}
\begin{tabular}{|l|l|l|l|l|l|}\hline
{\bf Cases}& $\bf x$&$\bf y$&$\bf z$&$\bf (x\circ y) \circ (z \circ (x\circ y))$&$\bf ((x \circ (y\circ z))\circ x) \circ y$\\\hline
{\bf Case 1:}& $g$&$h$&$k$&$gh\cdot (k\cdot gh)$&$(g \cdot hk)g \cdot h$\\\hline
{\bf Case 2:}& $g$&$h$&$kt$&$(gh \cdot ((gh)^{*}\cdot k))t$&$(h^{*}(g^{*}(g \cdot hk)))t$\\\hline
{\bf Case 3:}& $g$&$ht$&$k$&$g_{0}(gh \cdot h^{*}g^{*})k$&$g_{0}(g^{*}(g\cdot k^{*}h) \cdot h^{*})$\\\hline
{\bf Case 4:}& $gt$&$h$&$k$&$g_{0}((h^{*}g) \cdot ((h^{*}g)^{*} \cdot k^{*}))$&$g_{0}((k^{*}h^{*}\cdot g)g^{*})h$\\\hline
{\bf Case 5:}& $gt$&$ht$&$k$&$g_{0}g_{0}(gh^{*} \cdot (k\cdot gh^{*}))$&$g_{0}g_{0}((g\cdot h^{*}k)g \cdot h^{*})$\\\hline
{\bf Case 6:}& $gt$&$h$&$kt$&$g_{0}((h^{*}g \cdot k^{*})\cdot (h^{*}g))t$&$g_{0}(h^{*} \cdot (g \cdot (hk)^{*})g)t$\\\hline
{\bf Case 7:}& $g$&$ht$&$kt$&$g_{0}((gh \cdot k^{*}) \cdot gh)t$&$g_{0}(((g \cdot hk^{*})g \cdot h)t)$\\\hline
{\bf Case 8:}& $gt$&$ht$&$kt$&$g_{0}g_{0}(gh^{*} \cdot hg^{*})k)t$&$g_{0}((kh^{*} \cdot g)g^{*}\cdot h)t$\\\hline
\end{tabular}
\end{center}
Note that cases $1,5$ and $7$ follow from \eqref{RIF1}, and case 6 follows from \eqref{RIF2}.

For the case 2, we have
\begin{align*}
(\underbracket[0.75pt]{gh \cdot ((gh)^{*}\cdot k)})t
&\stackrel{\eqref{A2}}{=}((gh \cdot (gh)^{*})k)t
=((\underbracket[0.75pt]{gh\cdot h^{*}g^{*}})k)t
\stackrel{\eqref{A3}}{=}(((g\cdot hh^{*})g^{*})k)t\\ 
&=((gg^{*}\cdot hh^{*})k)t 
=(g^{*}g \cdot (h^{*}h\cdot k))t
=(h^{*} \cdot (\underbracket[0.75pt]{g^{*}g \cdot hk}))t\\ 
&\stackrel{\eqref{A2}}{=}(h^{*}(g^{*}(g \cdot hk)))t.
\end{align*}

For case 3, we have
\begin{align*}
g_{0}(\underbracket[0.75pt]{gh \cdot h^{*}g^{*}})k
&\stackrel{\eqref{A3}}{=}g_{0}(g(hh^{*})\cdot g^{*})k 
=g_{0}(hh^{*}gg^{*})k
=g_{0}(gg^{*}(\underbracket[0.75pt]{k\cdot hh^{*}})) \\
&\stackrel{\eqref{A1}}{=}g_{0}(gg^{*}(kh \cdot h^{*}))
=g_{0}((\underbracket[0.75pt]{g^{*}g \cdot k^{*}h})h^{*})
\stackrel{\eqref{A2}}{=}g_{0}(g^{*}(g\cdot k^{*}h) \cdot h^{*}).
\end{align*}

For case 4, we have
\begin{align*}
g_{0}(\underbracket[0.75pt]{(h^{*}g) \cdot ((h^{*}g)^{*} \cdot k^{*})})
&\stackrel{\eqref{A2}}{=}g_{0}((h^{*}g \cdot (h^{*}g)^{*})k^{*}) 
=g_{0}((\underbracket[0.75pt]{h^{*}g \cdot g^{*}h})k^{*})
\stackrel{\eqref{A3}}{=}g_{0}((h^{*} \cdot gg^{*})h \cdot k^{*})\\
&=g_{0}((gg^{*} \cdot hh^{*})k^{*})
=g_{0}((\underbracket[0.75pt]{k^{*} \cdot h^{*}h})\cdot gg^{*})
\stackrel{\eqref{A2}}{=}g_{0}((k^{*}h^{*} \cdot h)\cdot gg^{*})\\
&=g_{0}(\underbracket[0.75pt]{k^{*}h^{*}\cdot gg^{*}})h 
\stackrel{\eqref{A1}}{=}g_{0}((k^{*}h^{*}\cdot g)g^{*})h.
\end{align*}

For case 8, we have
\begin{align*}
g_{0}g_{0}((\underbracket[0.75pt]{gh^{*} \cdot hg^{*}})k)t
&\stackrel{\eqref{A3}}{=}g_{0}g_{0}((g(h^{*}h)\cdot g^{*})k)t 
=g_{0}g_{0}((g^{*}g \cdot h^{*}h)k)t
=g_{0}g_{0}((k \cdot h^{*}h)\cdot gg^{*})t \\
&=g_{0}g_{0}((\underbracket[0.75pt]{kh^{*} \cdot gg^{*}})h)t
\stackrel{\eqref{A2}}{=}g_{0}g_{0}((kh^{*} \cdot g)g^{*}\cdot h)t 
\end{align*} 
The argument for IP is similar to Theorem \ref{Cconstruction} and is left to the reader.
\end{proof}
\section{Connections between the extended Chein and extended de Barros-Juriaans constructions}
\label{connections}
We now focus our attention on combining the two constructions from Theorems \ref{Cconstruction} and \ref{dBJconstruction}.  We note that most of the following computations are straightforward, and are therefore left to the reader.
\begin{proposition}
$\quad$ Let $Q$ be a semiautomorphic IP loop and let $g_{0}\in Z(Q)$.  Then $g_{0}\in Z(Q\cup Qt,\circ)$ in either construction.
\label{P1}
\end{proposition}
\begin{proof}
Suppose $(Q\cup Qt,\circ)$ has the multiplication as in Theorem \ref{Cconstruction} and let $g_{0}\in Z(Q)$.  First note
\begin{equation*}
g_{0}\circ ht=(hg_{0})t=(hg_{0}^{*})t=ht\circ g_{0}.
\end{equation*}  
Hence, $g_{0}\in C(Q\cup Qt)$.  Now, let $x,y\in Q\cup Qt$.  It is enough to show $g_{0}\circ (x\circ y)=(g_{0}\circ x)\circ y$ and $x\circ (y\circ g_{0})=(x\circ y)\circ g_{0}$.  We have the following four cases:\\
\textbf{Case 1.}  Let $x=g,y=h$ for some $g,h\in Q$.  Thus
\begin{align*}
g_{0}\circ(g\circ h)&=g_{0}(gh)=(g_0g)h=(g_{0}\circ g)\circ h,\\ 
g\circ (h\circ g_{0})&=g(hg_{0})=(gh)g_0=(g\circ ht)\circ g_{0}.
\end{align*}
\textbf{Case 2.}  Let $x=g,y=ht$ for some $g,h\in Q$.  Thus
\begin{align*}
g_{0}\circ(g\circ ht)&=g_{0}\circ (hg)t=(hg\cdot g_{0})t=(h\cdot g_{0}g)t=g_{0}g\circ ht=(g_{0}\circ g)\circ ht,\\ 
g\circ (ht\circ g_{0})&=g\circ (hg_{0})t=(hg_{0}\cdot g)t=(hg\cdot g_{0})t=(hg)t\circ g_{0}=(g\circ ht)\circ g_{0}.
\end{align*}
\textbf{Case 3.}  Let $x=gt,y=h$ for some $g,h\in Q$.  Thus
\begin{align*}
g_{0}\circ(gt\circ h)&=g_{0}\circ (gh^{*})t=(g_{0}\cdot gh^{*})t=(gg_{0}\cdot h^{*})t=(gg_{0})\circ h=(g_{0}\circ gt)\circ h,\\ 
gt\circ (h\circ g_{0})&=gt\circ (hg_{0})=(g\cdot h^{*}g_{0})t=(gh^{*}\cdot g_{0})t=(gh^{*})t\circ g_{0}=(gt\circ h)\circ g_{0}.
\end{align*}
\textbf{Case 4.}  Let $x=gt,y=ht$ for some $g,h\in Q$.  Thus
\begin{align*}
g_{0}\circ(gt\circ ht)&=g_{0}\circ g_{0}(h^{*}g)=g_{0}\cdot g_{0}(h^{*}g)=g_{0}\cdot h^{*}(gg_{0})=(gg_{0})t\circ ht\\
&=(g_{0}\circ gt)\circ ht,\\ 
gt\circ (ht\circ g_{0})&=gt\circ (hg_{0})t=g_{0}\cdot (hg_{0})^{*}g=g_{0}\cdot (h^{*}g_{0}\cdot g)=(g_{0}\cdot h^{*}g)\circ g_{0}\\
&=(gt\circ ht)\circ g_{0}.
\end{align*}
The argument is similar if the multiplication is define as in Theorem \ref{dBJconstruction} and is left to the reader.
\end{proof}
\begin{proposition}
Let $Q$ be a semiautomorphic IP loop and $*$ an antiautomorphism of $Q$.  Reusing the symbol $*$, we extend $*$ on $Q\cup Qt$ as 
\begin{align*}
g^{*}&=g^{*},\\
(gt)^{*}&=gt.
\end{align*}
Then in either construction, the extended $*$ is an antiautomorphism of $(Q\cup Qt,\circ)$.
\label{P2}
\end{proposition}
\begin{proof}
Let $x,y\in Q\cup Qt$.  Then, using either multiplication, $4$ straightforward cases are needed to verify $(x\circ y)^*=y^*\circ x^*$ and are left for the reader.
\end{proof}
\begin{theorem}
Let $Q$ be a semiautomorphic IP loop with $g_{0}=1$ and $*$ an antiautomorphism (which we can extend by Proposition \ref{P2}).  Let $Q_{1}=(Q\cup Qs,\circ)$ with multiplication from Theorem \ref{dBJconstruction} and $Q_{2}=(Q\cup (Qs)t,\circ_2)$ where we apply the multiplication from Theorem \ref{Cconstruction} twice.  Then $Q_{1}\cong Q_{2}$.
\end{theorem}
\begin{proof}
Note the multiplication in $(Q_{2},\circ_2)$ is shown in Table \ref{table1}.
\begin{table}[h]
\centering
\begin{tabular}{|c|c|c|}\hline
$(Q_2,\circ_2)$&$h$&$(hs)t$\\\hline
$g$&$gh$&$((hg^*)s)t$\\\hline
$(gs)t$&$((gh)s)t$&$g^*h$\\\hline
\end{tabular}
\caption{Multiplication Table for $(Q_{2},\circ_{2})$.}
\label{table1}
\end{table}

Consider the bijection $\phi: Q_{1}\mapsto Q_{2}$ by 
\[
g\phi=g \quad (gs)\phi=(g^{*}s)t.
\]
To show $(x\circ_{1} y)\phi=x\phi\circ_{2} y\phi$ for all $x,y\in Q\cup Qt$, $4$ cases arise.  These are straightforward and left to the reader.
\end{proof}
\begin{proposition}
Let $Q$ be a semiautomorphic IP loop and $*$ an antiautomorphism of $Q$.  Let $c\in Z(Q)$ such that $c^{2}=1$ and $c^{*}=c$.  Then, reusing the symbol $*$, we extend $*$ on $Q\cup Qt$ as 
\begin{align*}
g^{*}&=g^{*},\\
(gt)^{*}&=c\cdot gt.
\end{align*}
Then in either construction, the extended $*$ is an antiautomorphism of $(Q\cup Qt,\circ)$.
\label{P3}
\end{proposition}
\begin{proof}
Again, $4$ cases are needed to verify $(x\circ y)^*=y^*\circ x^*$ for all $x,y\in Q\cup Qt$.  These are straightforward and left to the reader.
\end{proof}
\begin{theorem}
Let $Q$ be a semiautomorphic IP loop with $g_{0}\in Q$, $g_{0}^{2}=1$ and $*$ an antiautomorphism extending as in Proposition \ref{P3} with $c=g_{0}$.  Then doubling $Q$ twice first using the multiplication in Theorem \ref{dBJconstruction} followed by the multiplication in Theorem \ref{Cconstruction} is equivalent to doubling $Q$ twice using the multiplication in Theorem \ref{Cconstruction} twice.
\label{twice}
\end{theorem}
\begin{proof}
Let 
\[
Q_{1}=((Q\cup Qs)\cup(Q\cup Qs)t,\circ_{1})=(Q\cup Qs\cup Qt \cup (Qs)t,\circ_{1})
\] 
be the loop formed by first using the doubling construction in Theorem \ref{dBJconstruction} and then doubled again using the multiplication in Theorem \ref{Cconstruction}. Similarly, define 
\[
Q_{2}=((Q\cup Qs)\cup(Q\cup Qs)t,\circ_{2})=(Q\cup Qs\cup Qt \cup (Qs)t,\circ_{2})
\] 
where we double $Q$ twice using the multiplication in Theorem \ref{Cconstruction} twice. Then we have the following tables.

\begin{table}[h]
\centering
\begin{tabular}{|c|c|c|c|c|}
\hline $(Q_1,\circ_{1})$ &$h$   &$hs$  &$ht$  &$(hs)t$  \\ 
\hline  $g$     & $gh$        & $(gh)s$                	& $(hg)t$            &  $((g^{*}h)s)t$ \\ 
\hline  $gs$    & $(h^{*}g)s$ & $g_{0}(gh^{*})$        	& $((hg)s)t$         &  $(g_{0}\cdot hg^{*})t$\\ 
\hline  $gt$    & $(gh^{*})t$ & $[g_{0}\cdot((gh)s)]t$  & $g_{0}(h^{*}g)$    &  $((g^{*}h)s)$\\ 
\hline  $(gs)t$ & $(((hg)s)t)$& $gh^{*}t$ 				& $g_{0}((h^{*}g)s)$ &  $g_{0}(hg^{*})$\\ 
\hline 
\end{tabular} 
\caption{Multiplication Table for $(Q_{1},\circ_{1})$.}
\end{table}

\begin{table}[h]
\centering
\begin{tabular}{|c|c|c|c|c|}
\hline $(Q_2,\circ_{2})$ &$h$   &$hs$  &$ht$  &$(hs)t$  \\ 
\hline  $g$     & $gh$        & $(hg)s$                		& $(hg)t$            	&  $((hg^{*})s)t$ \\ 
\hline  $gs$    & $(gh^{*})s$ & $g_{0}(h^{*}g)$        		& $((gh)s)t$         	&  $(g_{0}\cdot g^{*}h)t$\\ 
\hline  $gt$    & $(gh^{*})t$ & $[g_{0}\cdot ((hg)s)]t$     & $g_{0}(h^{*}g)$   	&  $((hg^{*})s)$\\ 
\hline  $(gs)t$ & $(((gh)s)t)$& $(h^{*}g)t$ 			& $g_{0}((gh^{*})s)$ 	&  $g_{0}(g^{*}h)$\\ 
\hline 
\end{tabular} 
\caption{Multiplication Table for $(Q_{2},\circ_{2})$.}
\end{table}
Consider the bijection $\phi: Q_{1}\rightarrow Q_{2}$ defined as:
\[
g\phi=g \qquad (gs)\phi = (g^{*}s)t \qquad (gt)\phi = g_{0}\cdot gt \qquad ((gs)t)\phi = g^{*}s.
\]
Let $x,y\in Q\cup Qt$.  Then, to verify $(x\circ_{1} y)\phi=x\phi\circ_{2} y\phi$, $16$ cases are needed.  Again, the calculations are straightforward and left to the reader.
\end{proof}

\section{The constructions on other varieties of loops and examples}
\label{conclusion}
A loop $Q$ is a \emph{C-loop} if $C_{0}(Q)=Q$ (\emph{i.e.} $x\cdot y(yz)=(xy)y\cdot z$ holds for all $x,y,z$).  Since $C$-loops are closely related to Moufang and Steiner loops, it is natural to see examples of $C$-loops arise in this context.
\begin{theorem}
Let $Q$ be a semiautomorphic IP loop, $g_{0}\in Z(Q)$, and $*$ an involutory antiautomorphism of $Q$ such that $g_{0}^{*} = g_{0}, gg^{*} \in Z(Q)$ for every $g \in Q$.  Then, using either multiplication in Theorem \ref{Cconstruction} or \ref{dBJconstruction}, the following are equivalent:
\begin{itemize}
\item [(i)] $(Q\cup Qt,\circ)$ is commutative. 
\item [(ii)] $g^{*}=g$ for all $g\in Q$.
\end{itemize} 
Moreover, if either hold, then $g^{2}\in Z(Q)$ for all $g\in Q$ and $Q$ is a commutative $C$-loop.
\label{commutativecase}
\end{theorem}
\begin{proof}
Let $(Q\cup Qt,\circ)$ have the multiplication from Theorem \ref{Cconstruction}.  If $(Q\cup Qt,\circ)$ is commutative, then $(gh^{*})t=gt\circ h = h\circ gt = (gh)t$.  Letting $g=1$, we have the desired result.  Alternatively, $*$ is antiautomorphism of $Q$, and since $g^{*}=g$, we have $hg=h^{*}g^{*}=(gh)^{*}=gh$.  Hence, $(Q\cup Qt,\circ)$ is commutative.  Finally, if either $(i)$ or $(ii)$ holds, then $g^{2}=gg^{*}\in Z(Q)$ for all $g\in Q$.  Therefore, by \cite{chein3} and the fact that $Q$ is an IP loop, $Q$ is a commutative $C$-loop.
\end{proof}
\begin{example}
Let $Q$ be a Steiner loop of order $16$ with $|Z(Q)|=2$ (\emph{e.g.} \texttt{SteinerLoop(16,2)}). Let $g_{0}\in Z(Q)$, $g_0\neq 1$ and $g^{*}=g^{-1}$.  Then $(Q\cup Qt,\circ)$ with multiplication from Theorem \ref{Cconstruction} (or \ref{dBJconstruction}) is a commutative, non-Moufang, non-Steiner $C$-loop.
\end{example}
\begin{corollary}
Let $Q$ be a Steiner loop.  Then, for either multiplication in Theorem \ref{Cconstruction} or \ref{dBJconstruction}, if $(Q\cup Qt,\circ)$ is a Steiner loop, then $(Q\cup Qt,\circ)\cong Q\times Q$.
\end{corollary}
\begin{proof}
Let $(Q\cup Qt,\circ)$ have the multiplication from Theorem \ref{Cconstruction}.  If $(Q\cup Qt,\circ)$ is Steiner, hence commutative, then $g^{*}=g$ by Theorem \ref{commutativecase}.  Moreover, $1=gt\circ gt=g_{0}(g^{*}g)=g_{0}(g^{2})=g_{0}$.
\end{proof}
It is natural to ask what $(Q\cup Qt,\circ)$ would be if $Q$ started as a flexible loop satisfying \eqref{ARIF}.
\begin{theorem}  Let $Q$ be a flexible loop satisfying \eqref{ARIF}, $g_{0}\in Z(Q)$, and $*$ an involutory antiautomorphism of $Q$ such that $g_{0}^{*} = g_{0}, gg^{*} \in Z(Q)$ for every $g \in Q$.  For an indeterminate $t$, define multiplication $\circ$ on $Q \cup Qt$ by either multiplication in Theorem \ref{Cconstruction} or \ref{dBJconstruction}.  Then $(Q\cup Qt,\circ)$ is a flexible loop satisfying \eqref{ARIF}.
\label{ARIFconst}
\end{theorem}
\begin{proof}
Again, eight cases arise for either multiplication.  Using Lemmas \ref{L1} and \ref{TM}, it is straightforward to show $(z\circ x)\circ (y\circ x\circ y)=(z\circ(x\circ y\circ x))\circ y$ or $(y\circ x\circ y)\circ (x\circ z)=y\circ ((x\circ y\circ x)\circ z)$.
\end{proof}

\begin{example}
Let G be the Symmetric Group on 3 letters (\emph{i.e.} \texttt{SymmetricGroup(3)}).  Define $g_{0}=1$ and $g^{*} = g^{-1}$ for all $g\in G$.  Then 
$(G\cup Gt,\circ)$ with multiplication from Theorem \ref{Cconstruction} gives a Moufang loop of order 12 (\emph{i.e} \texttt{MoufangLoop(12,1)}), the smallest example of a nonassociative Moufang loop \cite{chein1, chein2}.  Moreover, $(G\cup Gt,\circ)$ with multiplication from Theorem \ref{dBJconstruction} gives a semiautomorphic IP loop of order 12, the smallest example that is non-Moufang and non-Steiner \cite{KPV}.
\label{smallest}
\end{example}
Note that for Moufang and Steiner loops, the nucleus is always a normal subloop \cite{bruck,PV}.
\begin{example}
Let $Q$ be a commutative Moufang loop of order $81$ (\emph{e.g.} \texttt{MoufangLoop(81,1)}).  Define $g_{0}=1$ and $g^{*} = g^{-1}$ for all $g\in G$.  Then $(Q\cup Qt,\circ)$ with either multiplication from Theorem \ref{Cconstruction} (or \ref{dBJconstruction}) gives a semiautomorphic IP loop where $N(Q) \ntrianglelefteq \ Q$.
\end{example}

\begin{acknowledgment}
Some investigations in this paper were assisted by the automated deduction tool \textsc{Prover9} and the finite model builder \textsc{Mace4} both developed by McCune \cite{PM}.  Similarly, all presented examples were verified using the GAP system \cite{GAP} together with the LOOPS package \cite{GAPNV}.
\end{acknowledgment}

\bibliographystyle{amsplain}
\bibliography{bib}
\end{document}